\documentclass[a4paper,12pt]{amsart}

\usepackage[latin1]{inputenc}
\usepackage[T1]{fontenc}
\usepackage[english]{babel}
\usepackage{mathrsfs}
\usepackage{amscd}
\usepackage{amsfonts}
\usepackage{amsmath}
\usepackage{amssymb}
\usepackage{amstext}
\usepackage{amsthm}
\usepackage{amsbsy}
\usepackage{xspace}
\usepackage[all]{xy}
\usepackage{graphicx}
\usepackage{url}
\usepackage{latexsym}

\newtheorem{prop}{Proposition}[section]

\newtheorem{thm}[prop]{Theorem}
\newtheorem{lemma}[prop]{Lemma}

\theoremstyle{remark}
\newtheorem*{fact}{Fact}

\DeclareMathOperator{\acl}{acl}

\def\be{\begin{equation}}
\def\ee{\end{equation}}

\newcommand*{\set}[1]{\{#1\}}
\newcommand{\C}{\mathbb C}

\newcommand{\N}{\mathbb N}

\newcommand{\Q}{\mathbb Q}

\DeclareMathOperator{\td}{t.d.}

\newcommand{\x}{{\bar x}}
\newcommand{\y}{{\bar y}}
\newcommand{\z}{{\bar z}}

\newcommand{\cv}{{\bar c}}

\title[]{A weak version of the strong exponential closure}
\author{P. D'Aquino}
\address{Dipartimento di Matematica e Fisica,  Universit\`a della Campania "L. Vanvitelli", Viale Lincoln 5, 81100 Caserta, Italy}
\email{paola.daquino@unicampania.it}
\author{A. Fornasiero}
\address{Dipartimento di Matematica e Informatica 'Ulisse Dini', Universit\`{a} di Firenze, Viale Morgagni, 67/a, 50134 Firenze, Italy}
\email{antongiulio.fornasiero@gmail.com}
\author{G. Terzo}
\address{Dipartimento di Matematica e Fisica, Universit\`a della Campania "L. Vanvitelli", Viale Lincoln 5, 81100 Caserta, Italy}
\email{giuseppina.terzo@unicampania.it}
\date{\today}

\begin{document}

\maketitle

\begin{abstract}
\noindent
Assuming Schanuel's Conjecture we prove that for any irreducible variety
$V \subseteq \Bbb C^n  \times (\Bbb C^{*})^n$ over  $\Bbb Q^{alg}$,
of dimension $n$, and with dominant projections on both the first $n$ coordinates and the last $n$ coordinates,
there exists a generic point $(\overline a, e^{\overline a}) \in V$. We obtain in this way
many instances of the Strong Exponential Closure axiom introduced by Zilber 
 in   \cite{zilber}.
\end{abstract}

\section{Introduction}

In \cite{zilber1}  Zilber conjectured that the complex exponential field is quasi-minimal, i.e. every subset of $\mathbb C$ definable in the language of rings expanded by the exponential function is either countable or co-countable. If the conjecture is true the complex exponential field should have  good geometric properties.

He introduced and studied a class of new exponential fields now known as Zilber fields via axioms of algebraic and geometrical nature. There are many novelties in his analysis, including a reinterpretation
of Schanuel's Conjecture in terms of Hrushovski's very general theory of predimension
and strong extensions, \cite{udi}. Zilber proved that  his axioms are uncountably categorical, and all models are quasiminimal.

Zilber conjectured that the one in cardinal $2^{\aleph_0}$ is $\mathbb C$ as exponential field. 
Comparing the complex exponential field and Zilber fields has been object of study in  \cite{marker}, \cite{dmt}, \cite{dmt2}, \cite{Shkop}, \cite{dmt1}, \cite{ayan}, \cite{mantova}.

In this paper we will analyze  one of the axioms introduced by Zilber,  the {\bf Strong Exponential Closure} (SEC),  in the complex exponential field. Modulo Schanuel's Conjecture, (SEC) is the only axiom  still unknown for $(\mathbb C, exp)$. Some instances of (SEC) for $(\mathbb C, exp)$ have been proved in   \cite{marker}, \cite{mantova}, \cite{dft}. Here we obtain a more general result which includes those in \cite{dft}.

\medskip
Let $G_n(\C) = \C^n \times (\C^{*})^n$ be the algebraic group. 
 Let $1\leq k\leq n$ and $M=(m_{ij})$  a $k \times n$  matrix of integers and $$[M]:G_n(\mathbb C)
 \rightarrow G_k(\mathbb C)$$  be the homomorphism given by $$(x_1, \ldots, x_n, y_1, \ldots, y_n)  \rightarrow  (x'_1, \ldots, x'_k, y'_1, \ldots, y'_k) $$
where $$x'_i = m_{i1} x_1 + \ldots + m_{in} x_n \mbox{ and } y'_i = y_1^{m_{i1}} \cdot \ldots \cdot y_n^{m_{in}},$$

\noindent for $i = 1, \ldots, k.$
We recall that a variety $V$ is {\it rotund} if for every nonzero matrix $M\in \mathcal M_{k\times n}(\mathbb Z)$ then $\dim ([M](V))\geq rank (M)$, i.e. all the images of $V$ under suitable homomorphisms are of large dimension.

A variety $V$ is {\it free} if $V$ does not lie inside any subvariety of the form either $\{(\bar x,\bar y): r_1x_1+\ldots +r_nx_n=b \}$ where $ r_i\in \mathbb Z, r_i\mbox{ not all $0$, } b \in \mathbb C $ or $\{(\bar x,\bar y):y_1^{ r_1}\cdot \ldots \cdot y_n^{ r_n}=b\}$ where $ r_i\in \mathbb Z, r_i\mbox{ not all $0$, } b \in \mathbb C^*$.

\smallskip

\noindent {\bf Strong Exponential Closure.} If $V \subseteq \Bbb C^n  \times (\Bbb C^{*})^n$ is a rotund and free algebraic  variety of dimension $n$, and  $\bar a $ is a finite tuple of elements of $\mathbb C$ then there is $\bar z \in \C^n$ such that $(\bar z, e^{\bar z}) \in V$, and  is generic in $V$ over $\bar a $, i.e. $\td_{\Q(\bar a)}(\bar z, e^{\bar z} )= \dim(V)$.

\medskip
 
The hypotheses of rotundity and freeness on the variety $V$  guarantee that the only relations among the coordinates of points in $V$ are those coming from $V$ itself and the laws of exponentiation.

\smallskip

We recall 

\smallskip

{\bf Schanuel Conjecture} (SC) Let $z_1,\ldots, z_n\in \mathbb C$. Then $$t.d._{\mathbb Q}(z_1,\ldots, z_n, e^{z_1}, \ldots ,e^{z_n})\geq l.d.(z_1,\ldots, z_n).$$

In this paper  assuming Schanuel's Conjecture we prove the Strong Exponential Closure for $(\mathbb C, exp)$ for certain varieties defined over $\mathbb Q^{alg}$.  We denote the projections on the first $n$ coordinates and  on the last $n$ coordinates by  $\pi_1:V \to \C^n$ and $\pi_2:V \to (\Bbb C^{*})^n$, respectively.  

\medskip

{\bf Main Result.}
(SC) 
Let $V \subseteq \Bbb C^n  \times (\Bbb C^{*})^n$ be an irreducible  variety defined over the
algebraic closure of $\Bbb Q$, such that $\dim V =n$, and both projections  $\pi_1$ and $\pi_2$  are dominant. Then there is a Zariski dense set of generic points $(\bar z, e^{\bar z})$ in V.

\medskip
We recall that $\pi_1$ and $\pi_2$  being dominant  means that $\pi_1(V)$ and $\pi_2(V)$  are Zariski dense in $\C^n$.   As observed in  \cite{bayskirby}, $\pi_1$ being dominant implies that $V$ is rotund and both projections being dominant imply that $V$ is free. 

There are examples of  free and rotund varieties with projections  not dominant, e.g.  $\{(x_1,x_2,y_1,y_2): x_2=x_1^2   \mbox{ and } y_2=y_1+1\}$.  

\smallskip

In Lemma 2.10 in \cite{dft}  (see also \cite{BM}) the existence of a Zariski dense set of solutions of  $V$ is  proved under the hypothesis that $\pi_1$ is dominant. No appeal to Schanuel's conjecture is necessary, and moreover there is no restriction on the set of parameters. 
For the new result on the existence of generic  solutions Schanuel's conjecture is crucial and there are restrictions on the set of parameters defining the variety $V$.

Recently,  
Bays and Kirby  in \cite{bayskirby}  proved the quasi-minimality of $(\mathbb C, exp)$  assuming  a weaker condition than the strong exponential closure, requiring   only the existence of a point $(\bar z, e^{\bar z} ) $ in  $V$ under the same hypothesis on the variety. No appeal to Schanuel's Conjecture is made. 

Some instances of quasi-minimality are known, e.g. if $X$ is a subset of $\mathbb C$ defined by either quantifier-free formulas or by $\forall \overline y(P(x, \overline y) = 0)$ where $P$ is a term in the language $\{ +, \cdot, 0,1, exp \}$ then $X$ is either countable or co-countable. Boxall in  \cite{boxall}  extends this result to  sets defined by an existential formula, 
$\exists \overline y(P(x, \overline y) = 0),$ where $P$ is a term in the language $\{ +, \cdot, 0,1, exp \}$ together with parameters from $\Bbb C.$

\section{Preliminaries}
We recall that the definable subsets of $\C^n$ (in the language of rings)  in the sense of model 
theory coincide with the constructible sets in algebraic geometry. 
We briefly review  some basic facts about the notion of dimension associated to a definable set in $\mathbb C^n$ which will be used in the proof of the main theorem, for details see \cite{van-den-Dries} and \cite{fornasiero}. 

We will always allow a finite or a countable set of parameters $P$. If not necessary we will not specify the set of parameters $P$. 

Every definable (with parameters in $P$) set  $X$ has a dimension
\[
\dim(X)= \max\{d: \exists \x \in X \mbox{  t.d.}_P(\x) = d\}.
\]

Let $\overline{X}^{Zar}$ denote the Zariski closure of $X$.  Then $\dim(X)= \dim (\overline{X}^{Zar}).$

Moreover,  for algebraically closed fields 
the model-theoretic algebraic closure ($acl$) coincides with the usual
field-theoretic algebraic closure.

\begin{fact} 1.
$\dim(X)$ is well-defined,  i.e.  it does not depend on the choice of the
set $P$ of parameters in the definition of $X$. 
\end{fact}

\begin{fact} 2.
Let $X$ be a definable set in $\C^n$.  The dimension of $X$ is $0$ iff $X$ is finite and nonempty.  We use the convention that the empty set has dimension $-1$.
\end{fact}

\noindent 
{\bf Notation.} Let $Y \subseteq \C^{n+m}$.  For  every $\x \in \C^n$ 
 $Y_\x=\{ \y \in \C^m: (\x,\y)\in Y\}$ and $Y^\y=\{ \x \in \C^n: (\x,\y)\in Y\}$. 

If $Y\subseteq  \C^{n+m}$ is definable then $(Y_{\x})_{\x \in \C^n}$ is a definable family. 

\begin{fact}  3.
Let $(Y_{\x})_{\x \in \C^n}$ be a definable family of subsets of $\C^m$. For every
$d \in \N$, the set $\set {\x \in \C^n:\dim(Y_{\x}) = d}$ is definable, with the same
parameters as $(Y_{\x})_{\x \in \C^n}$. For $d=0$ this gives that  $\{\x \in \C^n: Y_{\x} \text{ is finite}\}$ is also definable.

\end{fact}

\begin{fact}  4.
Let $(Y_{\x})_{\x \in \C^n}$ be a definable family of subsets of $\C^m$. Then  the family $({\overline Y_{\x}}^{Zar})_{\x \in \C^n}$ of the Zariski closures is still a definable family.  
\end{fact}

Let $\pi_1: Y \to \C^n$ and $\pi_2: Y \to \C^m$ be the projections on the first $n$ and the  last $m$ coordinates, respectively.

\begin{lemma}
\label{Fubini} 
Let $Y \subseteq \C^{n+m}$ be definable over $P$, and
$X := \pi_1(Y)$.
Assume that, for every $\x \in X$, $\dim(Y_\x) = d$.
Then, $\dim(Y) = \dim(X) + d$.
In particular, if $Y_\x$ is infinite for every $\x \in X$, then
$\dim(Y) > \dim(X)$.\end{lemma}

Notice that an equivalent  result holds in the case of $X:=\pi_2(Y)$. For the proof of the above lemma see \cite{van-den-Dries}.

Simple calculations give the following result.
\begin{lemma}
\label{fibre}
Let $Y \subseteq \C^n \times \C^m$ be a  definable set over $P$, such that $\dim(Y) \le n$.
Let $\cv \in \C^n$ be generic over $P$, i.e. $t.d._P(\cv)=n$.
Then, the fiber $Y_\cv := \set{\z \in \C^m: (\cv, \z) \in Y}$ is finite.
\end{lemma}

Brownawell and Masser in \cite{BM} develop a very powerful criterion for solvability of systems of exponential equations.  Proposition 2 in \cite{BM} implies the following result (see also \cite{dft}).

\begin{thm}
\label{esistenzasolvarieta}
Let $W \subseteq G_n(\C)$ be an irreducible algebraic variety such that $\pi_1(W)$ is Zariski dense in $\C^n$. Then, the set $\set{\overline a \in \C^n: (\overline a, e^{\overline a}) \in W}$ is Zariski dense in $\C^n$. 

\end{thm}

The hypothesis that $\pi_1(W)$ is Zariski dense is a non-trivial condition, and it implies that the variety is rotund. Theorem \ref{esistenzasolvarieta} states an even stronger property than the  Exponential-Algebraic Closedness for $(\mathbb C, exp)$ for  irreducible variety  $W$ with $\pi_1$ dominant.  Indeed,  there is a Zariski-dense sets of points $(\overline a, e^{\overline a})$ in $W$. 
A major problem is to replace the hypothesis that $\pi_1$ is dominant with much weaker ones like rotundity and freeness while still retaining the conclusion of the theorem.

Notice that no restriction is made on the coefficients of the polynomials defining $W$, and the result is independent from Schanuel's Conjecture.

\section{Strong Exponential Closure}
We now go back to analyze  Zilber's original axiom (SEC), i.e. we want to prove the existence of a  point in the variety $V$ of the form $(\overline a, e^{\overline a})$ which is generic in $V$. Assuming Schanuel's Conjecture we can prove (SEC) for algebraic varieties satisfying certain conditions. 

\begin{thm}
\label{mainresult}
(SC) 
Let $V \subseteq \Bbb C^n  \times (\Bbb C^{*})^n$ be an irreducible variety over the
algebraic closure of $\Bbb Q$ with $\dim V =n$. Assume that  both projections  $\pi_1$ and $\pi_2$ are  dominant. Then there is $\overline a\in \mathbb C^n$ such that $(\overline a, e^{\overline a})\in V$ and  $t.d._{\Bbb Q}(\overline a, e^{\overline a})=n$. In fact, the set $$\{ \overline a\in \mathbb C^n: (\overline a, e^{\overline a})\in V \mbox{ and  } t.d._{\Bbb Q}(\overline a, e^{\overline a})=n\}$$
 is  Zariski dense.  
\end{thm}

In the proof of Theorem \ref{mainresult} we will use the following known result. 

\

Let $M \in  \mathcal M_{m\times n}(\Bbb Z)$,   $L_M = \{\overline x \in \mathbb C^n: M \cdot \overline x = \overline 0\}$,  and  $T_M = \{\overline y\in (\mathbb C^*)^n : {\overline y}^M =\overline 1 \}.$ By ${\overline y}^M $ we denote the result of the exponential map applied to $M \cdot \overline x$, where $y_i=e^{x_i}$ for $i=1,\dots , n$.
\begin{fact}
 The  hyperspace  $L_M$ and the algebraic subgroup $T_M $ have the same dimension.
\end{fact}

\begin{proof}
Let $Z_M=
\{ \overline x \in \mathbb C^n : e^{M \cdot \overline x} = \overline 1 \} = \{ \overline x \in \mathbb C^n : e^{ \overline x} \in T_M\}=  \{ \overline x \in \mathbb C^n : M \cdot \overline x\in2\pi i\mathbb Z^m\}$.
 The algebraic subgroup $T_M$ is a closed differential submanifold in $(\mathbb C^*)^n$, and since $\exp $ is a local diffeomorphism  $Z_M$ is a differential submanifold of $\mathbb C^n$ of the same dimension as $T_M$.  Notice that $L_M$ is the tangent space of $Z_M$ at $\overline 0$, 
 and so $\dim (L_M)=\dim (Z_M)=\dim (T_M)$. 
\end{proof}

\noindent  {\it Proof of Theorem 3.1}. 
Let $U=\{ (\overline x,\overline y)\in V: |V_{\overline x}| \mbox{ and } |V^{\overline y}| \mbox{ are finite} \}$. Clearly, $U$ is  definable and Zariski dense in $V$. 
Let  $(\overline a, e^{\overline a}) \in U$, and suppose that $(\overline a, e^{\overline a})$ is not generic in $U$, i.e. $t.d._{\Bbb Q} (\overline a, e^{\overline a}) = m < n$.  The finite cardinality of $V_{\overline a}$ implies that all coordinates of the tuple $e^{\overline a}$ are  algebraic over ${\overline a}$, since they are in $\acl ({\overline a})$.  Exchanging ${\overline a}$ and $e^{\overline a}$ we have that each coordinate of the tuple ${\overline a}$ is algebraic over $e^{\overline a}$.  
 Hence, 
 \begin{equation}
 \label{tr.d.}
 m =  t.d._{\Bbb Q} (\overline a) =  t.d._{\Bbb Q} (\overline a, e^{\overline a}) =  t.d._{\Bbb Q} (e^{\overline a}).
 \end{equation}

Schanuel's Conjecture  implies $l.d.(\overline a) \leq t.d._{\Bbb Q} (\overline a, e^{\overline a}) = m < n$. By equation (\ref{tr.d.}) we can then conclude that  $l.d.(\overline a) =m$. Hence,  there exists a matrix $M \in \mathcal M_{(n-m)\times n}(\Bbb Z)$ of rank $n-m$ 
such that $M \cdot \overline a = \overline0$, which together with  its multiplicative version give the following hyperspace and torus:
$$L_M = \{\overline x : M \cdot \overline x = \overline 0\} \mbox{ and } T_M = \{\overline y : {\overline y}^M = \overline 1 \}.$$  As observed $\dim T_M = \dim L_M=m$.  So, $\overline a$ is generic in  $L_M$ and $e^{\overline a}$ is generic in $T_M.$ Then the non genericity of $ (\overline a, e^{\overline a}) $ in $U$ is witnessed either  by $\overline a$  or $e^{\overline a}$.

Without loss of generality we can assume that $T_M$ is irreducible. If not,  we consider the irreducible component  of $T_M$ containing $\overline{1}$ whose associated matrix we call $M'$.  By results on pages 82-83 in \cite{bombieri} the associate hyperspace $L_{M'}$ coincides with $L_M$.

For every $N \in \mathcal M_{(n-m)\times n}(\Bbb C),$ define
$$W_N = \{(\overline x, \overline y) \in U : \overline x \in L_N  \}.$$
Clearly, $W_N$ is definable, and so  $(W_N)_N$ is a definable family. 

If $N= M$ then $(\overline a, e^{\overline a}) \in W_M,$ and so $\dim W_M \geq \dim L_M.$  Moreover,  from the definitions of $U$ and $W_N$ it follows that   $\pi_1 $ restricted to  $W_M $ is finite-to-one. Therefore,   $\dim W_M = \dim \pi_1(W_M)$, and so   $\dim W_M \leq \dim L_M.$  Hence, $\dim W_M = \dim L_M.$

Let  $W'_M$  be the irreducible component of the Zariski closure of $W_M$ containing the point $(\overline a, e^{\overline a}).$  Since  $(\overline a, e^{\overline a})$ is generic  in $W'_M$, and $e^{\overline a} \in \pi_2(W'_M)\cap T_M$ is generic in  $\pi_2(W'_M)$  we have that   
$\pi_2(W'_M) \subseteq T_M.$  Hence,  the Zariski closure of the projection, ${\overline{\pi_2(W'_M)}}^{Zar},$ is contained in $T_M.$ Moreover, $e^{\overline a}$ is generic in  $T_M$, and this implies that  $T_M =  {\overline{\pi_2(W'_M)}}^{Zar}.$\\  
Let $(W_i^*)_{i\in I}$ (where $I$ is a definable set) be the definable family of all irreducible components of all ${(W_N)}_N$ for $N \in \mathcal M_{(n-m)\times n}(\Bbb C)$.  In particular, $W'_M$ is one of $W_i^*$  for some $i\in I$.  For each $i\in I$,  denote $S_i$  the Zariski closure of $\pi_2(W_i^*)$. 

Let
$$\mathcal U_m  = \{ S_i : S_i \mbox{ is a subgroup of  } (\mathbb C^*)^n \}.$$
Since $\mathcal U_m$ is a  countable definable family in $(\Bbb C^{*})^n$, and $\Bbb C$ is saturated then   $\mathcal U_m$ is either finite or uncountable.  Then $\mathcal U_m$ is necessarily  finite, and $T_M\in \mathcal U_m$. 

Let  $\mathcal U:= \mathcal U_1 \cup \ldots \cup \mathcal U_{n-1}$. Clearly, $\mathcal U$ is finite since each $\mathcal U_j$ is finite for $j\in \{1, \ldots , n-1\}$, so $\mathcal U=\{ H_1, \ldots, H_{\ell}\}.$
Let $T=H_1\cup \ldots \cup H_{\ell}$, and $C=\{  (\overline x,\overline y)\in U: \overline y\in T\}$. 
Then by Theorem \ref{esistenzasolvarieta}
the set  $X=\{(\overline a, e^{\overline a}) : (\overline a, e^{\overline a})\in U-C\}$ is not empty, Zariski dense in $U$ (and hence in $V$), and $t.d._{\Bbb Q} (\overline a, e^{\overline a}) = n$ for every $(\overline a, e^{\overline a})\in X$.  

\hfill $\Box$

\end{document}